\documentclass[a4paper,11 pt]{amsart}
\usepackage{hyperref}

\usepackage[usenames,dvipsnames,svgnames,table]{xcolor}
\usepackage[all]{xy}
 \usepackage{enumitem}
 \usepackage[normalem]{ulem}

\newcommand{\tsk}[1]{\textcolor{YellowOrange}}

\usepackage{tikz}
\usepackage{tikz-cd}

\makeatletter
\def\@endtheorem{\endtrivlist}
\makeatother

\usepackage[active]{srcltx}
\usepackage{amsmath}
\usepackage{amssymb}
\usepackage{amscd}
\usepackage{amsthm}
\usepackage[latin1]{inputenc}
\usepackage{mathrsfs}  

\usepackage{nicefrac}

\newtheorem{teo}{Theorem}[section]
\newtheorem{defin}[teo]{Definition}
\newtheorem{prop}[teo]{Proposition}
\newtheorem{cor}[teo]{Corollary}
\newtheorem{lemma}[teo]{Lemma}
\theoremstyle{definition}

\newtheorem{remark}[teo]{Remark}

\newtheoremstyle{dico}
 {\baselineskip}   
  {\topsep}   
  {}  
  {0pt}       
  {} 
  {.}         
  {5pt plus 1pt minus 1pt} 
  {}          
\theoremstyle{dico}
\newtheorem{say}[teo]{}
\numberwithin{equation}{section}

\newcommand{\ra}{\rightarrow}
\newcommand{\C}{\mathbb{C}}
\newcommand{\R}{\mathbb{R}}
\newcommand{\Zeta}{{\mathbb{Z}}}

\newcommand{\QQ}{{\mathbb{Q}}}

\newcommand{\alfa}{\alpha}

\newcommand{\vacuo}{\emptyset}

\newcommand{\La}{\Lambda}

\newcommand{\restr}[1]          {\vert_{#1}}
\newcommand{\Aut}{\operatorname{Aut}}

\newcommand{\End}{\operatorname{End}}

\newcommand{\spur}{\operatorname{Tr}}
\renewcommand{\setminus}{-}

\renewcommand{\phi}{\varphi}
\newcommand{\lds}{\ldots}

\newcommand{\cd}{\cdot}

\newcommand{\sx}{\langle}
\newcommand{\xs}{\rangle}

\newcommand{\ga}{\gamma}
\newcommand{\Ga}{\Gamma}

\newcommand{\gr}{\mathsf{g}}
\newcommand{\Fix}{\mathsf{Fix}}

\newcommand{\GL}{\operatorname{GL}}

\newcommand{\PP}{\mathbb{P}}

\renewcommand{\phi}             {\varphi}

\newcommand{\HH}{\mathfrak{H}}
\newcommand{\sieg}{\HH_g}

\newcommand{\Sp}                {\operatorname {Sp}}

\newcommand{\mm}{{\mathbf{m}}}
\newcommand{\jac}{\mathsf{T}^0_g}
\newcommand{\tor}{\mathsf{T}_g}

\newcommand{\ag}{\mathsf{A}_g}
\newcommand{\mg}{\mathsf{M}_g}
\newcommand{\zg}{\mathsf{Z}}
\newcommand{\datum}{{(\mm, G, \theta)}}

\newcommand{\Diff}{\operatorname{Diff}}
\newcommand{\Map}{\operatorname{Map}}

\newcommand{\ut}{U_t}

\begin{document}

\author{Paola Frediani, Matteo Penegini, Paola Porru}

\title[Shimura varieties via Galois  coverings of elliptic curves ]{Shimura varieties in the Torelli locus via Galois coverings of elliptic curves }

\address{Universit\`{a} di Pavia} \email{paola.frediani@unipv.it}
\address{Universit\`{a} di Milano}
\email{matteo.penegini@unimi.it}
\address{Universit\`a di Pavia}
\email{paola.porru01@ateneopv.it} 

\thanks{ 
} 

\thanks{ The first  author was partially supported by PRIN
  2012 MIUR ''Moduli, strutture geo\-me\-tri\-che e loro
  applicazioni''  and  by
  FIRB 2012 ''Moduli spaces and applications'' .   The  third author was partially
  supported by PRIN 2010 MIUR ``Geometria delle Variet\`a Algebriche".
 The three authors were partially supported by INdAM (GNSAGA).
} 

\maketitle

\begin{abstract}
We study Shimura subvarieties of $\ag$ obtained  from families of Galois coverings $f: C \rightarrow C'$ where $C'$ is a smooth complex projective curve of genus $g' \geq 1$ and $g= g(C)$. We give the complete list of all such families  that satisfy a simple sufficient condition that ensures that the closure of the image of the family via the Torelli map yields a Shimura  subvariety of $\ag$ for $g' =1,2$ and for all $g \geq 2,4$ and for $g' > 2$ and $g \leq 9$. In \cite{fgp} similar computations were  done in the case $g'=0$. Here we find 6 families of Galois coverings, all with $g' = 1$ and $g=2,3,4$ and we show that these are the only families with $g'=1$ satisfying this sufficient condition. We show that among these examples two families yield new Shimura subvarieties of $\ag$, while the other examples arise from certain Shimura subvarieties of $\ag$ already obtained as families of Galois coverings of $\PP^1$ in \cite{fgp}. Finally we prove that if a family satisfies this sufficient condition with $g'\geq 1$, then $g \leq 6g'+1$.  \end{abstract}

\tableofcontents{}

\section{Introduction}

The purpose of this paper is to continue the investigation started in \cite{fgp} of those special subvarieties of $\ag$ contained in the Torelli locus arising from families of Jacobians of Galois coverings $f: C \rightarrow C'$ where $C'$ is a smooth complex projective curve of genus $g' \geq 1$, $g=g(C)$. In \cite{fgp} the authors systematically studied families of Galois covering of $\PP^1$ following the previous work done by Moonen \cite{moonen-special} in the cyclic case and initiated in  \cite{shimura-purely-transcendental, mostow-discontinuous,
    dejong-noot, rohde} (see also the survey \cite[\S  5]{moonen-oort}). 
    
    More precisely, denote by $\ag$ the moduli space  of principally polarized abelian varieties of dimension $g$ over $\C$, by $\mg$ the moduli space of
  smooth complex algebraic curves of genus $g$ and by $j \colon \mg
  \ra \ag$ the period mapping or Torelli mapping. Set $\jac:=j
  (\mg)$ and call it the open Torelli locus.  The closure of $\jac$ in
  $\ag$ is called the \emph{Torelli locus} (see e.g.
  \cite{moonen-oort}) and is denoted by $\tor$. 
    
  The expectation formulated by Oort (\cite{oort-can}) is that for large enough genus $g$  there should not exist a positive-dimensional special subvariety $\zg$ of $\ag$, such that
  $\zg \subset \tor$ and $\zg\cap\jac\neq \vacuo$.  
  
  One reason for this expectation coming from differential geometry is that  a special (or Shimura) subvariety of $\ag$ 
  is totally geodesic with respect to the (orbifold) metric of $\ag$ induced by the symmetric metric on the Siegel space $\mathfrak{H}_g$ of which $\ag$ is a quotient by  $Sp(2g, 
 \Zeta)$. One expects the Torelli locus to be very curved and a way of expressing this is to say that it should not contain totally geodesic subvarieties. Important results in this direction were achieved in \cite{hain},  \cite{dejong-zhang}, \cite{toledo}, \cite{liu-yau-ecc}, \cite{lu-zuo-Mumford-prep}, \cite{chenluzuo},
  \cite{grushevsky-moeller-prep}.  In \cite{cfg} a study of the second fundamental form of the period map allowed to give an upper bound  for the possible dimension of a totally geodesic submanifold of $\ag$ contained in the Torelli locus.  This study was based on previous work on the second fundamental form of the period map done in \cite{cpt}, \cite{cf1}, \cite{cf2}. 
Moreover an important theorem of Moonen
  \cite{moonen-linearity-1} says that an algebraic totally geodesic
  subvariety of $\ag$ is special if and only if it contains a CM
  point, so the expectation formulated by Oort is both of geometric and  arithmetic nature. 
  See \cite[\S 4]{moonen-oort} for more details.

  On the other hand, as we mentioned above, for low genus $g \leq 7$
  there are examples of such $\zg$ and they are all constructed as families of Jacobians of Galois coverings of the line (see
  \cite{shimura-purely-transcendental, mostow-discontinuous,
    dejong-noot, rohde, moonen-special}, \cite[\S
  5]{moonen-oort} for the abelian Galois coverings, \cite{fgp} for the non abelian  case and for a complete list). 
  
  All the examples of families of Galois coverings constructed so far satisfy a sufficient condition to yield a Shimura subvariety that we briefly explain. 
  Consider  a Galois covering $f: C \rightarrow C' = C/G$, where $G \subset Aut(C)$ is the Galois group, $C'$ is a curve of genus $g'$.  Set $g = g(C)$, then one has a monomorphism of $G$ in the mapping class group $ \Map_g := \pi_0 ( \Diff^+ (C))$. The fixed point locus ${\mathcal T}_g^G$ of the action of $G$ on  the Teichm\"uller space 
   ${\mathcal T}_g$  is
  a complex submanifold of dimension $3g'-3 + r$ (see section 2). We consider its image $\mathsf{M}$ in $\mg$ and then the closure $\mathsf{Z}$ of the image of $\mathsf{M}$ in $\ag$ via the Torelli morphism. 
  
  Set $N:= \dim (S^2 H^0(C, K_C)) ^ G$, then the condition that we will denote by $(*)$ is that $N$ must be equal to the dimension of $\mathsf{Z}$, that is: 
  $$ (*) \ \ N = 3g'-3+r.$$
   In \cite{cfg} it is proven that this condition implies that the subvariety $\mathsf{Z}$ is totally geodesic  and in \cite{fgp} it is proven that in fact  it gives a Shimura subvariety in the case $g'=0$ and the same proof also works  if $g' >0$ as we remark in section 3. Moonen proved using arithmetic methods that condition $(*)$ is also necessary in the case of cyclic Galois coverings of $\PP^1$. 
 Results in this direction can also be found in \cite{mohazuo}. 
 
 In \cite{fgp} the authors gave the complete list of all the families of Galois coverings of $\PP^1$ of genus $g \leq 9$ satisfying condition $(*)$ and hence yielding Shimura subvarieties of $\ag$ contained in the Torelli locus. 
 
 In this paper we do the same for Galois coverings of curves of higher genus $g'$ and we find new examples when $g'=1$. We also prove that if $g'\geq1$, and the family satisfies $(*)$, then $g \leq 6g'+1$. This immediately implies that if $g'=1$ there are no  examples satisfying condition $(*)$ for $g \geq 8$. More precisely, we have  the following: 
 
 \begin{teo}
 \label{main}

 For all  $g \geq  2$ and $g' = 1$ there exist exactly  6 positive dimensional families of Galois coverings satisfying condition $(*)$, hence yielding Shimura subvarieties of $\ag$ contained in the Torelli locus.  
 
 Two of the 6 families yield new Shimura subvarieties, while the others yield Shimura subvarieties which have already been obtained as families of Galois coverings of $\PP^1$ in \cite{fgp}.

For all $g >3$ and  $g' =2$  there  do not exist positive dimensional families of Galois coverings satisfying condition $(*)$.

 For $g \leq 9$ and $g' > 2$ there  do not exist positive dimensional families of Galois coverings satisfying condition $(*)$.  
  \end{teo}

  \begin{teo}
  \label{xiaointro}
  If $g'\geq 1$ and we have a positive dimensional family of Galois coverings $f:C \ra C'$ with $g'= g(C')$ and $g = g(C)$ which satisfies condition  $(*)$, then $g \leq 6g'+1$.

  \end{teo}
  
  The proof of Theorem \ref{main} for $g \leq 9$, $g' \geq 1$ (and extend to $g\leq 13$ for $g'=2$) is done using   the  \verb|MAGMA| script that can be found at:
   
  \verb|users.mat.unimi.it/users/penegini/|

  \verb|publications/PossGruppigFix_Elliptic_v2.m|

  \smallskip

Theorem 
\ref{xiaointro} allows us to exclude the existence of any other family satisfying $(*)$ when $g'=1$ or $2$. The proof of Theorem \ref{xiaointro}  uses a result of Xiao (\cite{xiao}) and an argument analogous to the one used  in \cite{pietroxiao} to give a counterexample to a conjecture of Xiao on the relative irregularity of a fibration of a surface on a curve.

The 6 families  with $g'=1$ satisfying $(*)$, that is $N =r$, are the following:

 $(1)$ $g = 2$, $G = \Zeta/2\Zeta$, $N=r=2$. 
 
 $(2)$ $g =3$, $G= \Zeta/2\Zeta$, $N=r=4$.
 
  $(3)$ $g =3$, $G= \Zeta/3\Zeta$, $N=r=2$.
  
   $(4)$ $g =3$, $G= \Zeta/4\Zeta$, $N=r=2$.
   
    $(5)$ $g =3$, $G= Q_8$, $N=r=1$.
    
     $(6)$ $g =4$, $G= \Zeta/3\Zeta$, $N=r=3$.
     
     Family $(2)$ and family $(6)$ give two new Shimura subvarieties, while the others yield Shimura subvarieties which have already been obtained as families of Galois coverings of $\PP^1$ in \cite{fgp}.
     
    More precisely: 
    
    $(1)$ gives the same subvariety as family $(26)$ of Table 2 in \cite{fgp} (this family was already found in \cite{moonen-oort}). 
    
    $(3)$ gives the same subvariety as family $(31)$ of Table 2 in \cite{fgp}.
    
   $(4)$ gives the same subvariety as family $(32)$ of Table 2 in \cite{fgp}.
     
  $(5)$ gives the same subvariety as family $(34)$ of Table 2 in \cite{fgp}.
  
  All the above families with $g \geq 3$ are not contained in the hyperelliptic locus.

   A complete description of the families is given in section 4.

In section 4 we also show with a simple explicit computation that the families  we found with  \verb|MAGMA|  indeed satisfy condition $(*)$,  using Eichler trace formula (Theorem \eqref{Eichler}). 

We also notice that in {\cite{kurikuri} a classification of all the representations of the actions of the possible groups $G$ on the space of holomorphic one forms of a curve of genus $g =3,4$ is given and also using their description one can verify that in genus $3,4$ our families are the only ones satisfying condition $(*)$ if $g'=1$.

We finally  observe that the new family that we find in genus $g=4$ is very interesting also because it is the same family used by Pirola in \cite{pietroxiao} to give the above mentioned counterexample to a conjecture of Xiao on the relative irregularity of a fibration of a surface on a curve.  

The paper is organised as follows. 

In section 2 we recall some basic facts on Galois coverings of curves and we explain the construction of the families. 

In section 3 we recall very briefly the definitions and results (mostly without proofs) on special subvarieties of $\ag$ of PEL type and we show how the condition $(*)$ implies that the family yields a Shimura subvariety  following \cite{fgp}. 

In section 4 we give the explicit description of the new examples of special subvarieties obtained as Galois coverings of a genus 1 curve and we prove  Theorem \ref{main} and Theorem \ref{xiaointro}. 

In section 5 we briefly describe what the \verb|MAGMA| script does and we give the link to the script. 

\medskip

{\bf Acknowledgement}  It is a pleasure to thank E. Colombo, A. Ghigi, G.P. Pirola and C. Gleissner for stimulating discussions.

\section{Galois coverings }
\label{covering-section}
\begin{say} \label{theo: Riemann} 
Fix a compact Riemann surface $Y$ of genus $g' \geq 0$.
Let $t: = (t_1, \lds, t_r)$ be an $r$-tuple of distinct points in
$Y$.  Let us set $\ut := Y\setminus \{t_1, \lds, t_r\}$ and choose a
base point $t_0 \in \ut$. There exists an
isomorphism $\pi_1(\ut, t_0 ) \cong \Ga_{g',r}:= \langle \alpha _1, \beta _1, \dots \alpha _{g'}, \beta _{g'}, \gamma _1, \dots \gamma _r   \ | \ \prod _1 ^r \gamma _i \prod _1 ^{g'} \left[ \alpha _j, \beta _j \right] = 1 \rangle$ given by the choice of a geometric basis of $\pi_1(\ut, t_0 )$ as follows: 

$\alpha_1,\beta_1,...,\alpha_{g'}, \beta_{g'}$ are simple loops in $Y \setminus  \{t_1, \lds, t_r\}$ which only intersect in $t_0$, whose homology classes in $H_1(Y, \Zeta)$ form a symplectic basis.  

Let $\tilde{\gamma_i}$ be an arc connecting $t_0$ with $t_i$  contained in $(Y \setminus \{\alpha_1,\beta_1,...,\alpha_{g'}, \beta_{g'}\}) \cup t_0$ and such that for $ i \neq j$, $\tilde{\gamma_i}$ and  $\tilde{\gamma_j}$ only intersect in $t_0$.  We also assume that $\tilde{\gamma_1}, ...\tilde{\gamma_r}$ stem out of $t_0$ with distinct tangents following each other in counterclockwise order. The loops $\gamma_1,...,\gamma_r$ are defined as follows: $\gamma_i$ starts at $t_0$, goes along $\tilde{\gamma_i}$ to a point near $t_i$, makes a small simple loop counterclockwise around $t_i$ and goes back to $t_0$ following $\tilde{\gamma_i}$. 

 If $f\colon C \longrightarrow Y$ is
a Galois cover with branch locus $t$, set $V := f^{-1}(\ut)$. Then
$f\restr{V} : V \ra \ut$ is an unramified Galois covering. Let $G$
denote the group of deck transformations of $f\restr{V}$.  Then there
is a surjective homomorphism $
\pi_1(\ut, t_0 ) \longrightarrow G$, which is well-defined up to
composition by an inner automorphism of $G$.  Since $\Ga_{g',r} \cong
\pi_1(\ut,t_0)$ we get an epimorphism $\theta : \Ga_{g',r} \ra G$. If $m_i$
is the local monodromy around $t_i$, set  $\mm=(m_1, \lds, m_r)$. 

\begin{defin}
  A \emph{datum} is a triple $\datum$, where $\mm :=(m_1, \ldots ,m_r)
  $ is an $r$-tuple of integers $m_i \geq 2$, $G$ is a finite group
  and $\theta : \Ga_{g',r} \ra G$ is an epimorphism such that $\theta
  (\ga_i)$ has order $m_i$ for each $i$.
\end{defin}

 Thus a Galois cover of $Y$
branched over $t$ gives rise -- up to some choices -- to a datum. The
Riemann's existence theorem ensures that the process can be reversed:
a branch locus $t$ and a datum determine a covering of $Y$ up to
isomorphism (see e.g. \cite[Sec. III, Proposition 4.9]{M95}).  
The genus $g$ of the Riemann surface $C$ is given by Riemann-Hurwitz formula:
 \begin{equation*}\label{form.RH} 2g - 2 = |G|\left(2g'-2 +
      \sum_{i=1}^r \left(1 - \frac{1}{m_i}\right)\right).
  \end{equation*}

We will 
 show that the process can be reversed also in families, namely that
to any datum is associated a family of Galois covers of a compact Riemann surface $Y$ of genus $g'$.
\end{say}

\begin{say}
  \label{family}
  In fact let $\datum$ be a datum.  
  Choose a point $[Y, t=(t_1,...,t_r), \psi]$ in the Teichm\"uller space  $\mathcal{T}_{g',r}$. This means that $Y$ is a compact Riemann surface of genus $g'$, $t=(t_1,...,t_r)$ is an r-tuple of points in $Y$ such that $t_i \neq t_j$ for $i \neq j$ and $\psi: \pi_1(\ut, t_0) \cong \Gamma_{g',r} $ is an isomorphism, where $t_0$ is a base point in $U_t$.   Using $\theta \circ \psi$, by the above, we get a $G$-cover $C_t
  \ra Y$ branched at the points $t_i$ with local monodromies
  $m_1,\dots,m_r$.  This yields a monomorphism of $G$ into the mapping
  class group $ \Map_g := \pi_0 ( \Diff^+ (C_t))$.  Denote by ${\mathcal T}_g^G$
  the fixed point locus of $G$ on the Teichm\"uller space ${\mathcal T}_g$. It is
  a complex submanifold of dimension $3g'-3 + r$, isomorphic to the
  Teichm\"uller space ${\mathcal T}_{g',r}$ (see e.g. \cite{baffo-linceo,gavino}).
  This isomorphism can be described as follows: if $(C,\phi) $ is a
  curve with a marking such that $[(C, \phi)] \in {\mathcal T}_g^G$, the
  corresponding point in ${\mathcal T}_{g',r}$ is $[(C/G, \psi, b_1, \lds, b_r)]$,
  where $\psi$ is the induced marking (see \cite{gavino}) and $b_1,
  \lds, b_r$ are the critical values of the projection $C \ra C/G$.

  We remark that on ${\mathcal T}_g^G$ we have a universal family $\mathcal{C}
  \to {\mathcal T}_g^G$ of curves with a $G$--action. It is simply the
  restriction of the universal family on ${\mathcal T}_g$.

  We denote by $\mathsf{M}\datum$ the image of ${\mathcal T}_g^G$ in $\mg$. It is an irreducible algebraic
  subvariety of the same dimension as ${\mathcal T}_g^G\cong {\mathcal T}_{g',r}$, i.e. $3g'-3+r$
  (see e.g. \cite{baffo-linceo,gavino}).  Applying the Torelli map to
  $\mathsf{M}\datum$ one gets a subset of $\ag$.  We let
  $\zg(\mm,G,\theta)$ denote the closure of this subset in $\ag$.  By
  the above it is an algebraic subvariety of dimension $3g'-3+r$.
\end{say}

\begin{say}
  Different data $\datum$ and $(\mm, G, \theta')$ may give rise to the
  same subvariety of $\mg$. This is related to the choice of the
  isomorphism $\pi_1(\ut, t_0 ) \cong \Ga_{g', r}$.  The change from one
  choice to another can be described using an action of the full mapping class
  group ${\rm Map}_{g',[r]}:= \pi_0({\rm Diff}^+(Y\setminus \{t_1, \ldots .t_r\}))$. This action is described for example in \cite[Proposition 1.14]{penegini2013surfaces}.  

One has an injection of ${\rm Map}_{g',[r]} $ in ${\rm Out}^+(\Ga_{g',r})$ described in \cite[Section 4]{Macl74}.
Thus we get an action of ${\rm Map}_{g',[r]}$ on the set of data $\datum$ up to inner automorphisms.  Also the group $\Aut(G)$ acts on the set
  of data by $\alfa \cd (\mm, G, \theta) : = (\mm, G, \alfa \circ
  \theta )$.

  The orbits of the group  $\langle {\rm Map}_{g',[r]}, \,  \Aut(G) \rangle$--action (\emph{Hurwitz's moves}) are called
  \emph{Hurwitz equivalence classes}. Data in the same class give rise
  to the same subvariety $\mathsf{M}\datum$ and hence to the same
  subvariety $\zg\datum \subset \ag$. For more details see
  \cite{penegini2013surfaces,baffo-linceo,birman-braids}.

\end{say}

\begin{say}
  \label{say-ssg}
  
  Consider a datum $\datum$ and set $x_i := \theta(\ga_i)$.  If $C$ is a curve with an action of a group $G$ and  with datum
  $\datum$, then the cyclic subgroups $\left\langle
    x_{i}\right\rangle$ and their conjugates are the non-trivial
  stabilizers of the action of $G$ on $C$.  
  
  Near the fixed points the action of the
  stabilizers  can be described in
  terms of the epimorphism $\theta$, see \cite[Theorem 7]{Ha71}.   
  In fact, suppose that an element
  $\gr \in G$ fixes a point $P \in C$ and let $m$ be the order of $\gr$.
  The differential $d\gr_P$ acts on $T_P C$ by multiplication by an
  $m$-th root of unity $\zeta_P(\gr)$. The action can be linearized in
  a neighbourhood of $P$, i.e.  there is a local coordinate $z$ centered in $P$, such that $\gr $
  acts as $ z \mapsto \zeta_{P}(\gr ) z$.  So $\zeta_P(\gr )$ is a
  primitive $m$-th root of unity.  (See also \cite[Cor. III.3.5
  p. 79]{M95}.)

  Denote by $\Fix(\gr)$ the set of fixed points of
  $\gr$  and  set $ \zeta_m = e^{2\pi i / m} $, $I(m) := \{\nu \in \Zeta: 1\leq \nu < m, \gcd(\nu, m) =1 \}.$
 For $\nu \in I(m)$ consider
  
  \begin{gather*}
    \Fix_\nu (\gr ) : = \{P \in C : {\bf g } P = P , \zeta_P (\gr ) =
    \zeta_m^{-\nu}\}.
  \end{gather*}

\end{say}

\begin{lemma}
  If $G \subseteq\Aut(C)$ and $\gr \in G$ has order $m$, then
  \begin{gather*}
    | \Fix_\nu (\gr ) | = | C_G(\gr ) | \cd \sum_{
      \substack{ 1 \leq i \leq      r , \\
        m | m_i ,\\
        \gr \sim_G x_i ^{m_i \nu / m}}} \frac{1}{m_i}.
  \end{gather*}
\end{lemma}
(Here $C_G(\gr)$ denotes the centralizer of $\gr$ in $G$ and $\sim_G$
denotes the equivalence relation given by conjugation in $G$.)  This
lemma follows from \cite[Theorem 7]{Ha71}, see also \cite[Lemma
11.5]{breuer}.

\begin{say}\label{say_rep}
  Given a $G$-Galois cover $C \rightarrow Y$ let $ \rho\colon G
  \longrightarrow {\rm GL}(H^0(C, K_C)) $ be the representation on
  holomorphic 1-forms given by $g \mapsto (g^{-1})^*$. Let $\chi_\rho$ be the character of
  $\rho$. Notice that up to equivalence the representation $\rho$ only
  depends on the data $(\mm, G, \theta)$.
  \end{say}

\begin{teo}
\label{Eichler}
[Eichler Trace Formula]

Let $\gr $ be an automorphism of order $m>1$ of a Riemann surface $C$
of genus $g>1$. Then
\begin{equation}\label{eq_EichlerFormula}
  \chi_\rho(\gr )=
  \spur(\rho({\gr}))=1+\sum_{P\in \Fix(\gr)}\frac{\overline{\zeta_P(\gr)}}{1-\overline{\zeta_P(\gr)}}.
\end{equation}  
\end{teo}
(See e.g.  \cite[Thm. V.2.9, p. 264]{FK}.)  Using the previous lemma one gets the following.
\begin{cor}
  \begin{equation}
    \label{carg}
    \chi_\rho(\gr )=
    1+
    |C_G(\gr)| 
    \sum_{\nu \in I(m)} 
    \biggl\{\sum_{\substack{
        1 \leq i \leq      r , \\
        m | m_i ,\\
        \gr \sim_G x_i ^{m_i \nu / m}}}
    \frac{1}{m_i}  \biggl \} \frac{\zeta_m^\nu}{1-\zeta_m^\nu}.
  \end{equation}  
\end{cor}

\begin{say}
  Let $\sigma\colon G \rightarrow {\rm GL}(V)$ be any linear
  representation of $G$ with character $\chi_\sigma$. Denote by
  $S^2\sigma$ the induced representation on $S^2 V$ and by
  $\chi_{S^2\sigma}$ its character.  Then for $x \in G$
  \begin{equation}\label{eq_symchar}
    \chi_{S^2\sigma}  
    (x)=\frac{1}{2}\bigl (\chi_\sigma(x)^2+\chi_\sigma(x^2)\bigr).
  \end{equation}
  (See e.g. \cite[Proposition 3]{S}).

  We are only interested in the multiplicity $N$ of the trivial
  representation inside $S^2\rho$. We remark that since the
  representation $\rho$ only depends on the datum $(\mm, G, \theta)$,
  the same happens for $N$.  Using the orthogonality relations and
  \eqref{eq_symchar}, $N$ can be computed as follows:
  \begin{gather}
    \label{N}
    N= (\chi_{S^2\rho}, 1) = \frac{1}{|G|} \sum_{x \in G }
    \chi_{S^2\rho}(x) = \frac{1}{2|G|} \sum_{x \in G } \bigl (
    \chi_\rho(x^2) + \chi_\rho(x)^2 \bigr).
  \end{gather}
\end{say}

\section{Special subvarieties}
\label{Shimura-section}

\begin{say}
  \label{VHS}
  Fix a rank $2g$ lattice $\La$ and an alternating form $Q : \La
  \times \La \ra \Zeta$ of type $(1,\lds, 1)$.  For $F$ a field with
  $\QQ \subseteq F \subseteq \C$, set $\La_F : = \La\otimes_\Zeta F$.
  The Siegel upper half-space is defined as follows
  \cite[Thm. 7.4]{kempf-abelian-theta}:
  \begin{gather*}
    \sieg:= \{J \in \GL (\La_ \R) : J^2 = - I, J^* Q = Q, Q(x,Jx) >0,
    \ \forall x \neq 0 \}.
  \end{gather*}
  The group $\Sp(\La, Q)$ acts on $\sieg$ by conjugation and $\ag =
  \Sp(\La, Q) \backslash \sieg $.  This space has the structure of
  a smooth algebraic stack and also of a complex analytic orbifold. The
  orbifold structure is given by the properly
  discontinuous action of $\Sp(\La, Q) $ on $\sieg$.  Throughout the
  paper we will work with $\ag$ with this orbifold structure.  Denote
  by $A_J$ the quotient $\La_ \R / \La$ endowed with the complex
  structure $J$ and the polarization $Q$.  On $\sieg$ there is a
  natural variation of rational Hodge structure, with local system
  $\sieg \times \La _ \QQ$ and corresponding to the Hodge
  decomposition of $\La_ \C$ in $\pm i$ eigenspaces for $J$.  This
  descends to a variation of Hodge structure on $\ag$ in the orbifold
  or stack sense.
\end{say}

\begin{say}
 A special
    subvariety $\zg \subseteq\ag$ is by definition a Hodge locus of
  the natural variation of Hodge structure on $\ag$ described above.
  For the definition of Hodge
  loci for a variation of Hodge structure we refer to \S 2.3 in \cite{moonen-oort}.  
  Special subvarieties contain a dense set of CM points and they are
  totally geodesic \cite[\S 3.4(b)]{moonen-oort}. Conversely an
  algebraic totally geodesic subvariety that contains a CM point is a
  special subvariety \cite[Thm. 4.3]{moonen-linearity-1}.  In this paper we will only deal with the special varieties of
    PEL type, whose definition is as follows (see \cite[\S
  3.9]{moonen-oort} for more details). 
  Given $J\in \sieg$, set
  \begin{gather}
    \label{endq}
    \End_\QQ (A_{J}) := \{f\in \End_\QQ(\La_ \QQ): Jf=fJ\}.
  \end{gather}
  Fix a point $J_0 \in \sieg$ and set $D:= \End_\QQ (A_{J_0})$.  The
  \emph{PEL type} special subvariety $\zg (D)$ is defined as
  the image in $\ag$ of the connected component of the set $\{J \in
  \sieg: D \subseteq\End_\QQ(A_J)\}$ that contains $J_0$.
\end{say}

Recall the following results proven in \cite[Section 3]{fgp}.

   \begin{prop}
    \label{bert}
    Let $G\subseteq\Sp(\La, E)$ be a finite subgroup. Denote by
    $\sieg^G$ the set of points of $\sieg$ that are fixed by $G$. Then
    $\sieg^G$ is a connected complex submanifold of $\sieg$.
  \end{prop}
       Set
  \begin{gather}
    D_G:=\{ f\in \End_\QQ (\La_ \QQ) : Jf=fJ, \ \forall J \in
    \sieg^G\}.
    \label{def-DG}
  \end{gather}

\begin{lemma}
  If $J \in \sieg^G$, then $D_G \subseteq\End_\QQ (A_J)$ and the
  equality holds for $J$ in a dense subset of $\sieg^G$.
\end{lemma}
\begin{prop}
  \label{prop-Bert}
  The image of $\sieg^G$ in $\ag$ coincides with the PEL subvariety
  $\zg (D_G)$.
\end{prop}

\begin{lemma}
  \label{dimensione}
  If $J \in \sieg^G $, then $\dim \sieg^G = \dim \zg(D_G) = \dim (S^2
  \La _\R)^G$ where $\La_\R$ is endowed with the complex structure
  $J$.
\end{lemma}

Recall that $N = \dim \left ( S^2H^0(C,K_C)\right )^G $ and that
$\zg(\mm, G, \theta)$ is defined in \ref{family}.

\begin{teo}
  \label{criterio}
  Fix a datum $\datum$ and assume that
  \begin{gather}
  \label{bona}
    \tag{$\ast$} N = 3g'-3 +r.
  \end{gather}
  Then $\zg (\mm, G, \theta)$ is a special subvariety of PEL type of
  $\ag$ that is contained in $\tor$ and such that $\zg(\mm, G,
  \theta)\cap \jac \neq \vacuo$.
\end{teo}
\begin{proof}
  Let $\mathcal{C} \ra {\mathcal T}_g^G$ be the universal family as in \ref{family}.  For any
  $t\in {\mathcal T}_g^G$, $G$ acts holomorphically on $C_t$, so it maps
  injectively into $\Sp(\La, Q)$, where $\La = H_1(C_t,\Zeta)$ and $Q$
  is the intersection form.  Denote by $G'$ the image of $G$ in
  $\Sp(\La, Q)$. It does not depend on $t$ since it is purely
  topological.  Recall that the Siegel upper half-space $\sieg$ parametrizes
  complex structures on the real torus $\La _ \R / \La = H_1(C_t, \R)
  / H_1(C_t, \Zeta)$ which are compatible with the polarization $Q$.  The period
  map associates to the curve $C_t$ the complex structure $J_t$ on
  $\La _ \R$ obtained from the splitting $ H^1(C_t, \C) = H^{1,0}(C_t)
  \oplus H^{0,1}(C_t) $ and the isomorphism $H_1(C_t, \R)^* _ \C =
  H^1(C_t, \C)$.  The complex
  structure $J_t$ is invariant by $G'$, since the group $G$ acts holomorphically on $C_t$. This shows that $J_t \in
  \sieg^{G'}$, so the Jacobian $j(C_t) $ lies in $\zg(D_{G'})$. This
  shows that $\zg (\mm, G, \theta) \subseteq \zg(D_{G'})$. Since
  $\zg(D_{G'})$ is irreducible (see e.g. Proposition \ref{bert}), to
  conclude the proof it is enough to check that they have the same
  dimension. The dimension of $\zg(\mm, G, \theta)$ is $3g'-3 +r$, see \ref{family}.  By
  Lemma \ref{dimensione}, if $J \in \sieg^{G'}$, then $\dim
  \zg(D_{G'}) = \dim\sieg^{G'} = \dim ( S^2 \La _ \R ) ^{G'}$, where
  $\La _ \R$ is endowed with the complex structure $J$.  If $J$
  corresponds to the Jacobian of a curve $C$ in the family, then $(
  S^2 \La _ \R ) ^{G'}$ is isomorphic to the dual of $(S^2 H^0(C,
  K_C))^G$.  Thus $\dim \zg(D_{G'} )= N$ and \eqref{bona} yields the
  result.
\end{proof}

\section{Examples of special subvarieties in the Torelli locus}

\label{examples-section}

\begin{say}
  \label{list}
  In this section we prove Theorem \ref{main} and Theorem \ref{xiaointro}. First we give the list of all the  families of Galois
  covers of a curve of genus 1 satisfying property $(*)$, hence giving rise to special subvarieties of $\ag$. 
  For them we now give a presentation of the Galois group $G$ and an
  explicit description of a representative of an epimorphism $$\theta: \Gamma_{1, r} = \langle \alpha, \beta, \gamma_1,... ,\gamma _r   \ | \ \gamma_1...\gamma_r  \alpha\beta\alpha^{-1} \beta^{-1} = 1 \rangle \rightarrow G$$

  (we use the same notation as in \S
\ref {theo: Riemann} and \S \ref {say-ssg}).

\bigskip
  \begin{center}
    Genus $2$
  \end{center}
  \smallskip

  \begin{itemize}[labelindent=\parindent,leftmargin=*]
  \item [(1)]  $G= \Zeta/2\Zeta = \sx z \ | \ z^2 = 1\xs$.   \\
  $\mm=(2, 2)$
    $\theta : \Gamma_{1, 2} \ra \Zeta/2\Zeta$,\\
    $\theta(\gamma_1) = \theta(\gamma_2)= z$, $\theta(\alpha) = \theta(\beta) = 1$.

   \end{itemize}

  \smallskip
  \begin{center} {Genus $3$}
  \end{center}
 
  \smallskip

\begin{itemize}[labelindent=\parindent,leftmargin=*]
  \item [(2)]  $G= \Zeta/2\Zeta = \sx z \ | \  z^2 = 1\xs$.   \\
  $\mm=(2, 2,2,2)$
    $\theta : \Gamma_{1, 4} \ra \Zeta/2\Zeta$,\\
    $\theta(\gamma_i) = z$, $\forall i =1,...,4$, $\theta(\alpha) = \theta(\beta) = 1$. 
    
      \item[(3)]  $G= \Zeta/3\Zeta = \sx z \ | \ z^3 = 1\xs$.   \\
  $\mm=(3, 3)$
    $\theta : \Gamma_{1, 2} \ra \Zeta/3\Zeta$,\\
    $\theta(\gamma_1) = z,  \theta(\gamma_2)= z^2$, $\theta(\alpha) = \theta(\beta) = 1$. 

 \item[(4)]  $G= \Zeta/4\Zeta = \sx z \ | \  z^4 = 1\xs$.   \\
  $\mm=(2, 2)$
    $\theta : \Gamma_{1, 2} \ra \Zeta/4\Zeta$,\\
    $\theta(\gamma_1) =  \theta(\gamma_2)= z^2$, $\theta(\alpha) = \theta(\beta) = 1$. 

 \item[(5)]  $G = Q_8 = \langle g_1, g_2, g_3 : g_1 ^2 = g_2 ^2 = g_3, g_3 ^2 = 1, g_1 ^{-1} g_2 g_1 = g_2 g_3 \rangle$ \\
  $\mm=(2)$
    $\theta : \Gamma_{1, 1} \ra Q_8$,\\
    $\theta(\gamma_1) = g_3$,  $\theta(\alpha) =g_2$,  $\theta(\beta) = g_1$. 

   \end{itemize}

 \smallskip
  \begin{center} {Genus $4$}
  \end{center}
 
  \smallskip

\begin{itemize}[labelindent=\parindent,leftmargin=*]
  \item [(6)]  $G= \Zeta/3\Zeta = \sx z \ | \ z^3 = 1\xs$.   \\
  $\mm=(3, 3,3)$
    $\theta : \Gamma_{1, 3} \ra \Zeta/3\Zeta$,\\
    $\theta(\gamma_i) = z$, $\forall i =1,2,3$, $\theta(\alpha) = \theta(\beta) = 1$. 
\end{itemize}

\end{say}
\begin{say}\label{say_HWMoves}

First of all we will show that in each of the above cases, once we fix the group $G$ and the ramification $\mm$, all the possible data $(\mm, G, \theta)$ belong to the same Hurwitz equivalence class and hence give rise to the same subvariety of $\ag$. Let us prove this for all the cases, giving the details only for some cases, being the others very similar. 

Case (1).  The group $\Zeta/2\Zeta\cong\langle z\rangle$ has only one element of order $2$, this forces  $\theta(\gamma_i)$ to be equal to $z$ for $i=1,2$. Recall that the action of the mapping class group ${\rm Map}_{1,[2]}$  on $\Gamma_{1, 2}$ is generated (up to inner automorphism) by the seven moves described in \cite[Proposition 1.14]{penegini2013surfaces} -- from where we also borrow the notation. In particular, since $\Zeta/2\Zeta$ is abelian, the move induced by $t_{\xi^1_{1,1}}$  is given by $\theta(\alpha) \mapsto \theta( \gamma_1 \alpha)$ and the identity on the other generators,  while  the move  induced by $t_{\xi^2_{1,1}}$ is given by $\theta(\beta) \mapsto \theta(\gamma_1\beta)$ and the identity on the other generators. This gives that the systems of generators $\langle \theta(\alpha), \theta(\beta); \, \theta(\gamma_1), \theta(\gamma_2)\rangle= \langle z,z; \, z,z \rangle$ and $\langle 1,1; \, z,z \rangle$ are Hurwitz equivalent. In addition the moves $t_{\delta_1}$   is given by $\alpha \mapsto \alpha\beta^{-1}$ and the identity on the other generators, and $t_{\tilde{\delta}_1}$   is given by $\beta \mapsto \beta\alpha$ and the identity on the other generators. This moves yields at once that  systems of generators $\langle z,z; \, z,z \rangle$,  $\langle 1,z; \, z,z \rangle$, $\langle z,1; \, z,z \rangle$  are Hurwitz equivalent. 

The proof for the cases (2) and (4) is the same as the one for case (1) with obvious changes.  

Case (3). Since $\Zeta/3\Zeta$ is abelian its commutator subgroup is trivial, thus  up to automorphisms we can choose $\theta(\gamma_1)=z$ and $\theta(\gamma_2)=z^2$. Then we proceed as in case (1) with obvious changes. Notice that this proof can be adjusted to fit for case (6).

Case (5). Since there is only one element of order $2$ in $Q_8$, $\theta(\gamma_1)=g_3$. Up to simultaneous conjugation the image of the pair  $(\alpha, \, \beta)$ by $\theta$ is one of the following $(g_1, \, g_2)$, $(g_2, \, g_1)$,  $(g_1, \, g_1g_2)$, $(g_1g_2, \, g_1)$, $(g_1g_2, \, g_2)$ and $(g_2, \, g_1g_2)$.   Using only the moves  $t_{\delta_1}$ and  $t_{\tilde{\delta}_1}$,  described above, and the automorphisms of $Q_8$ we see that all the pairs are equivalent to $(g_2, \, g_1)$. Therefore all the systems of generators are Hurwitz equivalent to $\langle g_2, g_1; \, g_3 \rangle$. Notice that this proof can be found also in \cite[Proposition 5.9]{matteo2011}. Indeed, in that article this very covering is used to construct a new surface of general type with $p_g=q=2$.  

\end{say}
\begin{say}
\label{paragone}
We will now show that the families listed above do in fact verify condition $(*)$ and hence yield special subvarieties of $\ag$. Nevertheless we will show that only two of them, namely family  $(2)$ and $(6)$ give rise to new special subvarieties of $\ag$, while the others have already been obtained as families of Galois covers of $\PP^1$. 

Let us explain this. Assume that a family of Galois coverings of genus 1 curves with Galois group $G$ satisfying $(*)$ yields the same Shimura subvariety of $\ag$ of dimension $s$ as one of those obtained via a family of Galois coverings of $\PP^1$ satisfying condition $(*)$. Then each covering $\phi:X \rightarrow X/G$ of the family of covers of genus one curves has the property that the curve $X$ also admits an action of a group $K \subset Aut(X)$ such that $X/K \cong \PP^1$ and we have: $dim(S^2H^0(K_X)^G) = dim(S^2H^0(K_X)^K) = s$.  So each curve $X$ of the family admits an action of a group $\tilde{G} \subset Aut(X)$ containing both $G$ and $K$ such that $X/\tilde{G} \cong \PP^1$ and since $S^2H^0(K_X)^{\tilde{G}} \subset S^2H^0(K_X)^K$, also the family of coverings $\psi: X \rightarrow X/\tilde{G} \cong \PP^1$ satisfies condition $(*)$ and we have the following commutative diagram:

\begin{equation}
\label{factorization} 
\xymatrix{ X \ar[r]^{\phi} \ar[rd] _{\psi} & X/G \ar[d] ^{\sigma} \\
              & X/\tilde{G} \cong \mathbb{P}^1
}
\end{equation}
So we can assume that $G \subset K$.

Since all the families of Galois coverings of $\PP^1$ satisfying $(*)$ in genus less than 10 have already been found in \cite{fgp} Table 2, it will suffice to compare our families with the ones listed there.
\end{say}

\begin{say}
Example $(2)$. 
Clearly $\theta$ is an epimorphism. 
Now we want to show that the sufficient condition $ (*)$ is satisfied, so we compute the number $N$  using \eqref{N}.  Eichler trace formula \eqref{carg} immediately yields $\chi_{\rho} (z) = -1$, and since $\chi_{\rho}(1) =  g = 3$ we have $N= 4$, which coincides with the number of critical values, and so with the dimension of the family. This proves that our family of Galois coverings is special.

Since the family just constructed has dimension $4$, it has  bigger dimension than any possibile family given as Galois covering of $\mathbb{P}^1$ satisfying $(*)$. In fact looking at the table of all possible special varieties presented as Galois coverings of $\mathbb{P}^1$ satisfying $(*)$  of genus $g \leq 9$ we see that none of these has dimension greater than $3$ (see Table 2 of \cite{fgp}). This proves that the family gives a new special subvariety contained in the Torelli locus. It also follows that the family is not contained in the hyperelliptic locus. In fact, if every curve $C$ of the family were hyperelliptic, one could consider the group $H$ generated by $G$ and the hyperelliptic involution $\sigma$. The quotient $C/H \cong \PP^1$, so we would obtain a family of Galois coverings of $\PP^1$ which clearly still verifies condition $(*)$, and this does not exist, as we just pointed out.

\end{say}

\begin{say}
Example $(6)$. 
Clearly $\theta$ is an epimorphism. 

Eichler trace formula \eqref{carg} immediately yields $\chi_{\rho} (z) = \zeta_3$, $\chi_{\rho} (z^2) = \bar{\zeta}_3$ and since $\chi_{\rho}(1) =  g = 4$, by \eqref{N}  we have $N= 3$, which coincides with the number of critical values, and so with the dimension of the family. This proves that our family is special.

As noted before, to conclude we have to check that the family does not yield the same Shimura subvariety of $\ag$ as one obtained via a family of Galois coverings of $\mathbb{P}^1$ already known to be special. Nonetheless, looking at Table 2 in \cite{fgp} one checks there are no families of Galois covers of $\mathbb{P}^1$ satisfying condition $ (*)$ with dimension greater or equal than $3$ admitting $\mathbb{Z}/3\mathbb{Z}$ as a proper subgroup of the Galois group. By \eqref{paragone}, this proves  that the family  gives a new special subvariety contained in the Torelli locus. It also follows by the same argument as in the previous example that it is not contained in the hyperelliptic locus.

\begin{remark}
We observe that this family is interesting also for another reason, in fact it is the same family used by Pirola in \cite{pietroxiao} to construct a counterexample to a conjecture of Xiao on the relative irregularity of a fibration of a surface on a curve. 

\end{remark}
\end{say}

\begin{say}
Example $(1)$. Clearly $\theta$ is an epimorphism. 

By Eichler trace formula  \eqref{carg} we find $\chi_{\rho} (z) = 0$, and using \eqref{N}   we obtain $N= 2$, which coincides with the number of critical values, and so with the dimension of the family, therefore  our family  is special.

We will now show that this family yields the same subvariety in $\ag$ as family $(26)$ in Table 2 of \cite{fgp} (this family was already found in \cite{moonen-oort}).

Let us recall the description of this family. It is a family of Galois coverings of $\PP^1$ with Galois group $\tilde{G} = \mathbb{Z}/2\mathbb{Z} \times \mathbb{Z}/2\mathbb{Z} = \langle x,y |  \, x^2 = y^2 = 1, \, x y = y x \rangle $, with ramification data $(2,2,2,2,2)$ and with epimorphism   $\tilde{\theta} : \Gamma_{0,5}= \langle \delta_1,..., \delta_5 \ | \ \prod_{i=1,...,5} \delta_i= 1 \rangle \rightarrow \tilde{G}$
given by 
$$\tilde{\theta}  (\delta _1 ) = x, \;\tilde{\theta} (\delta _2) = x, \; \tilde{\theta}  (\delta _3) = x, \;\tilde{\theta}  (\delta _4) = y, \;\tilde{\theta}  (\delta_5) = xy.$$

We want to prove that for any covering of this family $\psi : X \rightarrow X/\tilde{G} \cong \mathbb{P}^1$, $X$ also admits an action by  a subgroup $G \cong \Zeta/2\Zeta$ of $\tilde{G}$   such that  the quotient map $\phi: X \rightarrow E \cong X/G$,  belongs our family $(1)$  and we have a diagram as in \eqref{factorization}. 

Consider the cyclic subgroup $G \cong \langle y | \, y ^2 = 1 \rangle < \tilde{G}$.

Looking at ramification data, we see that all stabilizer subgroups of $\tilde{G}$ have order $2$ and looking at the epimorphism $\tilde{\theta}$  we see that the stabilizer subgroup associated to the fourth branch point $q_4$ is $\langle y 
\rangle = G$. This implies that points in $\psi ^{-1}(q_4) = \lbrace p_1, p_2 \rbrace$ are critical points for the action of $G$ as well. Moreover they cannot belong to the same fiber with respect to the action of $G$ since every element of $G$ stabilizes both $p_1$ and $p_2$. To conclude note that every other stabilizer subgroup for the critical points of $\psi$ does not contain any non trivial element of $G$, so $q_4$ is the only branch point of $\phi$. This proves that the map $\phi$ has exactly $2$ critical values which are the images of $p_1$ and $p_2$ by $\phi$ and the ramification is $\mm=(2,2)$. So we have a family with the same group and the same ramification as in family $(1)$ and hence by the unicity argument given in \ref{say_HWMoves} we conclude that it gives the same special subvariety in $\ag$ as the one given by family $(1)$.

Concluding, the special subvariety given by family $(1)$ gives the same special subvariety obtained as a family of Galois coverings of $\mathbb{P}^1$ via $\mathbb{Z}/2\mathbb{Z} \times \mathbb{Z}/2\mathbb{Z}$ corresponding to family $(26)$ of Table 2 of \cite{fgp} and already studied in \cite{moonen-oort}.

\end{say}

\begin{say}
Example (3).  It is clear that $\theta$ is an epimorphism. Using  Eichler trace formula we immediately obtain $\chi_{\rho}(z) = \chi_{\rho}(z^2)=0$ and by  \eqref{N}   we obtain $N= 2$, which coincides with the number of critical values, and so with the dimension of the family, therefore the family  is special.

We claim that this family yields the same Shimura subvariety of $\ag$ as family $(31)$  in Table 2 of \cite{fgp}. Let us describe this family of Galois coverings of $\PP^1$ as in 4.1 of \cite{fgp}. The Galois group $\tilde{G}$ is isomorphic to the symmetric group $S_3$, $\tilde{G} = \langle x,y |  \, y^2 = x ^3 = 1, \, y ^{-1} x y = x^2 \rangle$,  $\mm=(2,2,2,2,3)$ and the epimorphism $\tilde{\theta}: \Gamma_{0,5} = \langle \delta_1,..., \delta_5 \ | \ \prod_{i=1,...,5} \delta_i= 1 \rangle \rightarrow \tilde{G}$
is given by 
$$\tilde{\theta}  (\delta _1 ) = xy, \;\tilde{\theta} (\delta _2) = x^2y, \; \tilde{\theta}  (\delta _3) = y, \;\tilde{\theta}  (\delta _4) = xy, \;\tilde{\theta}  (\delta_5) = x^2.$$
We will show that every covering $\psi:X \rightarrow X/ \tilde{G} \cong \PP^1$ of this family  also  admits a $G=\Zeta/3\Zeta$-action such that $X/G$ has genus 1, the map $\phi: X \rightarrow X/G$ is one of the coverings of our family (3) and we have a factorisation as in \eqref{factorization}. 
In fact, set $G = \langle x | \, x ^3 = 1 \rangle < \tilde{G}$ that is the only cyclic subgroup of order 3 on $\tilde{G}$. Looking at the stabilisers of the action of $\tilde{G}$, we see that the two critical points of $\psi$ in the fibre over the critical value $q_5$ have both $G$ as stabiliser, hence they are critical points also for the action of $\phi$ and they are mapped by $\phi$ in two different critical values. All the other critical points of $\psi$ have stabilisers of order 2, hence they are not critical values for the map $\phi$. So the map $\phi$ has ramification $(3,3)$ and by the unicity argument \ref{say_HWMoves} we can assume that $\phi: X \rightarrow X/G$ belongs to our family $(3)$. 

Concluding, the special subvariety given by family $(3)$ gives the same special subvariety obtained as the family  $(31)$ of Galois coverings of $\mathbb{P}^1$ via $S_3$ found in \cite{fgp}. This family is not contained in the hyperelliptic locus (see the proof of theorem 5.3 of \cite{fgp}). 
\end{say}

\begin{say}
Example $(4)$. 
 It is clear that $\theta$ is an epimorphism. Using  Eichler trace formula we immediately obtain $\chi_{\rho}(z) =\chi_{\rho}(z^3)  = 1$, $\chi_{\rho}(z^2) = -1$ and by  \eqref{N}   we obtain $N= 2$, which coincides with the number of critical values, and so with the dimension of the family, therefore the family  is special.

We will now show that this family yields  the same Shimura subvariety of $\ag$ as family $(32)$  in Table 2 of \cite{fgp}. Let us describe this family of Galois coverings of $\PP^1$ as in 4.1 of \cite{fgp}.

The Galois group $\tilde{G}$ is isomorphic to the dihedral  group $D_4$, $\tilde{G} = \langle x,y |  \, y^2 = x ^4 = 1, \, y ^{-1} x y = x^3 \rangle$,  $\mm=(2,2,2,2,2)$ and the epimorphism $\tilde{\theta}: \Gamma_{0,5} = \langle \delta_1,..., \delta_5 \ | \ \prod_{i=1,...,5} \delta_i= 1 \rangle \rightarrow \tilde{G}$ is 
given by 
$$\tilde{\theta}  (\delta _1 ) = xy, \;\tilde{\theta} (\delta _2) = x^2y, \; \tilde{\theta}  (\delta _3) = x^2, \;\tilde{\theta}  (\delta _4) = x^2y, \;\tilde{\theta}  (\delta_5) = x^3y.$$

As above we want to show that every covering $\psi:X \rightarrow X/ \tilde{G} \cong \PP^1$ of this family  also  admits a $G=\Zeta/4\Zeta$-action such that $X/G$ has genus 1, the map $\phi: X \rightarrow X/G$ is one of the coverings of our family (4) and we have a factorisation as in \eqref{factorization}.

We can identify $G \cong \langle x | \, x ^4 = 1 \rangle < \tilde{G}$.

The stabilizer subgroups for the action of $\tilde{G}$ of the critical points over  the third branch point $q_3$ are all given by the center $H = \langle x^2 \rangle$ of $\tilde{G}$ which is contained in $G$. This implies that points in $\psi ^{-1}(q_3) = \lbrace p_1, p_2, p_3, p_4 \rbrace$ are critical points for the action of $G$ as well.

Morover the four points $p_1,..,p_4$ are partitioned in exactly two orbits for the action of $G$, hence they give rise to two critical values of the map $\phi$. Finally, observing that the other $4$ conjugacy classes of stabilizers for $D_4$ do not contain nontrivial elements belonging to $G$, we conclude that the action of $G< D_4$ has ramification data $(2,2)$ and by \ref{say_HWMoves} we can assume that it gives a covering belonging to our family $(4)$. 

Concluding, the special variety given by family $(4)$ gives the same special variety obtained as the family $(32)$ of Galois coverings of $\mathbb{P}^1$ via $D_4$ found in \cite{fgp}. This family is not contained in the hyperelliptic locus (see the proof of theorem 5.3 of \cite{fgp}).

\end{say}

\begin{say}
Example (5). One easily checks that $\theta$ is an epimorphism. Using  Eichler trace formula we find that the trace of every non zero element different from $g_3$ is equal to $1$, and that $\chi_{\rho} (g_3) = -1$.  By  \eqref{N}   we obtain $N= 1$, which coincides with the number of critical values, and so with the dimension of the family, therefore the family  is special.

We will now show that this family yields  the same Shimura subvariety of $\ag$ as family $(34)$  in Table 2 of \cite{fgp}. Let us describe this family of Galois coverings of $\PP^1$ as in 4.1 of \cite{fgp}.

The Galois group  is $$\tilde{G}=(\mathbb{Z}/4\mathbb{Z} \times \mathbb{Z}/2\mathbb{Z}) \rtimes (\mathbb{Z}/2\mathbb{Z}) \cong$$
$$\langle y_1, y_2, y_3 |  \, y_1^2 = y_2 ^2 = y_3 ^4= 1, \, y_2y_3 = y_3 y_2, 
y_1 ^{-1}y_2 y_1 = y_2 y_3 ^2, \, y_1 ^{-1}y_3 y_1 = y_3 \rangle,$$
$\mm=(2,2,2,4)$ and the epimorphism $$\tilde{\theta}: \Gamma_{0,4} = \langle \delta_1,..., \delta_4 \ | \ \prod_{i=1,...,4} \delta_i= 1 \rangle \rightarrow \tilde{G}$$
is given by 
$$\tilde{\theta}  (\delta _1 ) = y_1, \;\tilde{\theta} (\delta _2) = y_1 y_2 y_3 ^3, \; \tilde{\theta}  (\delta _3) = y_2 y_3 ^2, \;\tilde{\theta}  (\delta _4) =  y_3 ^3.$$

Observe that the conjugacy classes of the nontrivial elements of $\tilde{G}$ are: 
\begin{equation*}\begin{aligned}
& \qquad \mbox{order } 2: \; \lbrace y_1, y_3 ^2 y_1 \rbrace , \, \lbrace y_2, y_3 ^2 y_2 \rbrace , \, \lbrace y_3^2 \rbrace ,  \, \lbrace y_2 y_3 y_1, y_2 y_3 ^3 y_1 \rbrace ,\\
&\qquad \mbox{order } 4: \; \lbrace y_3 \rbrace , \, \lbrace y_3 ^3 \rbrace \, \lbrace y_2 y_3, y_2 y_3 ^3\rbrace \, \lbrace y_2 y_1, y_2 y_3^2 y_1\rbrace \, \lbrace y_3 y_1, y_3 ^3 y_1 \rbrace , \\
\end{aligned}\end{equation*}

As above we want to show that every covering $\psi:X \rightarrow X/ \tilde{G} \cong \PP^1$ of this family  also  admits a $G=Q_8$-action such that $X/G$ has genus 1, the map $\phi: X \rightarrow X/G$ is one of the coverings of our family (5) and we have a factorisation as in \eqref{factorization}.

In order to prove that the factorization holds, first at all we have to check that $G$ is isomorphic to a subgroup of $\tilde{G}$. 
One easily checks that the following map $$i: Q_8 \rightarrow (\mathbb{Z}/4\mathbb{Z} \times \mathbb{Z}/2\mathbb{Z}) \rtimes (\mathbb{Z}/2\mathbb{Z})$$
 \begin{equation*}
g_1 \mapsto y_2 y_3, \; g_2 \mapsto y_2 y_1, \; g_3 \mapsto y_3 ^2.
\end{equation*}
yields an injective homomorphism that identifies $G=Q_8$ with a proper subgroup of $(\mathbb{Z}/4\mathbb{Z} \times \mathbb{Z}/2\mathbb{Z}) \rtimes (\mathbb{Z}/2\mathbb{Z})$.

As before, to conclude that the two families are in fact the same one, we have to study their stabilizer subgroups. Looking at the epimorphism $\tilde{\theta} : \Gamma_{0,4} \rightarrow \tilde{G}$ we see that the stabilizer subgroup associated to the fourth branch point $q_4$ is the normal subgroup $K:= \langle y_3 ^3 \rangle = \lbrace 1, y_3, y_3 ^2, y_3 ^3 \rbrace$.The subgroup $H = \lbrace 1, y_3^2 \rbrace$ of $G$ is clearly contained in  $ K= \langle y_3 ^3 \rangle$. This implies that points in $\psi ^{-1}(q_4) = \lbrace p_1, p_2, p_3, p_4 \rbrace$ are critical points for the action of $G$ as well. Up to a permutation of the $p_i$'s, we see that $\tilde{G}$ acts this way on the fiber:
\begin{equation*}\begin{aligned}
&p_2 = y_1 (p_1)= y_1 y_3 (p_1) = y_1 y_3 ^2( p_1) = y_1 y_3 ^3 (p_1),\\
&p_3 = y_2 (p_1) = y_2 y_3 (p_1) = y_2 y_3 ^2 (p_1) = y_2 y_3 ^3 (p_1),\\
&p_4 = y_2 y_1 (p_1) = y_2 y_1 y_3 (p_1) = y_2 y_1 y_3 ^2( p_1) = y_2 y_1 y_3 ^3 (p_1).\\
\end{aligned}\end{equation*}
If we consider the action of $G$ we get:
\begin{equation*}\begin{aligned}
&p_2 = y_1 y_3 (p_1) = i(g_2^{-1} g_1)(p_1) = y_1 y_3 ^3(p_1) = i(g_1^{-1}g_2)(p_1), \\
&p_3 = y_2 y_3 (p_1) = i(g_1)(p_1)= y_2 y_3 ^3(p_1) = i(g_1^{-1})(p_1), \\
&p_4= y_2 y_1 (p_1) = i(g_2)(p_1)= y_2 y_1 y_3 ^2 (p_1)= i(g_2g_3)(p_1).\\
\end{aligned}\end{equation*}
So $p_1,p_2,p_3,p_4$ are 4 ramification points that map to the same critical value for the map $\phi$ and they have multiplicity 2.  To conclude we have to prove that these are all the critical points of $\phi$.  But this is actually true, because none of the stabilizer subgroups of the critical points of $\psi$ that are mapped to the first three critical values includes any subgroup of $G$.
Concluding, by the unicity argument \ref{say_HWMoves}, the special variety given by the family $(5)$ gives the same special variety obtained as the family $(34)$ of Galois coverings of $\mathbb{P}^1$ via $(\mathbb{Z}/4\mathbb{Z} \times \mathbb{Z}/2\mathbb{Z}) \rtimes (\mathbb{Z}/2\mathbb{Z})$ found in \cite{fgp}.  This family is not contained in the hyperelliptic locus (see the proof of theorem 5.3 of \cite{fgp}).

\end{say}

Finally we prove Theorem \ref{xiaointro} that we restate here as follows:

\begin{teo}
\label{xiao}

 If $g'\geq 1$ and we have a positive dimensional family of Galois coverings $f:C \ra C'$ with $g'= g(C')$ and $g = g(C)$ which satisfies condition  $(*)$, then $g \leq 6g'+1$.  
 
In particular, for $g \geq 8 \, ({\rm resp.}\, 14)$ there do not exist positive dimensional  families of Galois coverings with $g' =1 \,  ({\rm resp.} \, 2)$ and which satisfy condition $(*)$. 
\end{teo}

\proof

The idea of the proof is the following: if such a family exists, with the same method used by Pirola in section 2 of \cite{pietroxiao} one constructs a fibration $S \ra B$ of a surface $S$ on a curve $B$, whose general fibre has genus $g$ and whose relative irregularity is at least $g-g'$. Then we apply  Corollary 3 of \cite{xiao}.

In fact, assume that $\mathsf{M} := \mathsf{M}(\mm, G, \theta)$ is as usual the variety parametrising elements of such a family for a given datum $(\mm, G, \theta)$. Every point $\mathsf{p} \in  \mathsf{M}$ corresponds to an isomorphism class of a  curve $C$ of genus $g$ admitting $G$ as a subgroup of $Aut(C)$, whose quotient $C':= C/G$ has genus $g'\geq 1$ and whose monodromy is given by $\theta$. Denoting by $f: C \ra C'$ the Galois covering, to such a point one can associate the abelian variety $W = J(C)/f^*(J(C'))$, which is isogenous to the Prym variety of the covering.  The abelian variety $W$ has a polarisation $\Theta$ and we denote by $\mathsf{A}{_{g-g'}}(\Theta)$ the moduli space of polarised abelian varieties of dimension $g-g'$ with the given type of polarisation.  Denote by $\Psi: \mathsf{M} \ra \mathsf{A}{_{g-g'}}(\Theta)$ the map associating to $\mathsf{p}$ the polarised variety $[W, \Theta]$. The variety $W$ inherits from $C$ the automorphism group $G$.  The differential of the map $\Psi$ at the point $\mathsf{p} \in \mathsf{M}$ is a map 
$$d\Psi_{\mathsf{p}}: H^1(C, T_C)^G \ra S^2H^{0,1}(W)$$
and its image is contained in $ (S^2H^{0,1}(W))^G$, since this  is the space of infinitesimal deformations of $(W, \Theta)$ that preserve the action of $G$. So if we denote by $\mathsf{P} \subset \mathsf{A}{_{g-g'}}(\Theta)$ the image of $\Psi$ , the tangent space $T_{[W, \Theta]}\mathsf{P}$ of $\mathsf{P}$ at $[W, \theta]$ is contained in $ (S^2H^{0,1}(W))^G$. The dual of the differential gives a map: 
$$d\Psi_{\mathsf{p}}^*: (S^2H^{1,0}(W))^G \ra H^0(C, 2K_C)^G.$$
Observe that $H^0(K_C)= H^0(K_C)^G \oplus H^0(K_C)^-$, where the space of invariants $H^0(K_C)^G \cong H^0(C',K_{C'})$ has dimension $g'$, and the complement $ H^0(K_C)^- \cong H^{1,0}(W)$. Therefore  we have 

  \begin{multline}
    \label{W}
    (S^2H^0(K_C))^G \cong S^2H^0(K_{C'}) \oplus (S^2H^{0}(K_C)^-)^G \cong \\
   \cong S^2H^0(K_{C'}) \oplus (S^2H^{1,0}(W))^G   \end{multline}

The dual of the differential $d\Psi_{\mathsf{p}}^*: (S^2H^{0}(K_C)^-)^G \ra H^0(C, 2K_C)^G$ is given by the multiplication map and since $(S^2H^{0}(K_C)^-)^G \subset (S^2H^{0}(K_C))^G$ and by our assumption $(*)$ the multiplication map 
$$(S^2H^{0}(K_C))^G \ra H^0(C, 2K_C)^G$$ is an isomorphism, we conclude that $d\Psi_{\mathsf{p}}^*$ is injective. Hence $d\Psi_{\mathsf{p}}$ is surjective and $T_{[W, \Theta]}\mathsf{P} = (S^2H^{0,1}(W))^G$.  By \eqref{W} its dimension is equal to $N-\frac{g'(g'+1)}{2}$, where $N =\dim (S^2H^0(K_C))^G$ is the dimension of  $\mathsf{M}$, by our condition $(*)$. So for a general point $[W,\Theta] \in \mathsf{P}$, $\dim\Psi^{-1}(W, \Theta) = \frac{g'(g'+1)}{2} \geq 1$, thus we can find a curve $Y \subset \Psi^{-1}(W, \Theta)$ contained in $\mathsf{M}_g$. 
Denote by $\overline{Y}$ its closure in $\overline{\mathsf{M}}_g$. So we get a family of curves of genus $g$,  $h': S' \ra B'$ such that $\overline{Y}$ is the image of the modular map $B' \ra \overline{\mathsf{M}}_g$, $b' \mapsto [h'^{-1}(b')]$. By resolving singularities and taking pullbacks we get a smooth surface $S$, a smooth curve $B$  and a map $h:S \ra B$. Up to a base change  we can assume that $h$ has a section $\eta$. So if we take the Zariski open subset $U$ of $B$ of points having nonsingular fibres, we can use the section $\eta$ to take the Abel-Jacobi maps $A_{\eta(t)}: C_t \ra J(C_t)$, $t \in U$,  compose them with the projections $J(C_t) \ra W$ and obtain mappings: $\phi_t: C_t \ra W$, $\forall t \in U$.  Using the pull-backs $\phi_t^*: H^1(W, \QQ) \ra H^1(C_t, \QQ)$ we get an injection of $H^1(W, \QQ)$ in $H^0(B, R^1h_*{\QQ})$. By the Leray spectral sequence we  identify $H^0(B, R^1h_*{\QQ})$ with the cockernel of the map $h^*: H^1(B, \QQ) \ra H^1(S, \QQ)$, thus we have 
$$\dim H^1(S, \QQ) - \dim H^1(B, \QQ) \geq \dim H^1(W, \QQ) = 2(g-g').$$
So if we denote by $q = h^0(S, \Omega^1_S)$ and by $b$ the genus of the curve $B$ we have $q-b \geq g-g'$. Since by construction the family is not isotrivial, we can apply Corollary 3 of \cite{xiao}, which says that $q-b \leq \frac{5g+1}{6}$ and so we get  $g - g'\leq q-b \leq \frac{5g+1}{6}$, hence $g\leq 6g' +1$.

Clearly if $g'=1$ this implies $g \leq 7$.

\qed\\

Using the above Theorem, to conclude the proof of Theorem \ref{main} it only remains to show that if $g \leq 7  \, ({\rm resp.}  \, 13)$ and $g'=1 \,  ({\rm resp.} \, 2)$ there does not exist any other family satisfying $(*)$ except for the 6 families described above  and if $g \leq 9$ and $g'>1$ there do not exist families satisfying property $(*)$. To do this we use the \verb|MAGMA| script briefly described below.


\section{Higher Genus}

A slightly modified version of the \verb|MAGMA| script used in \cite{fgp} enables us to check that the families given is Section 3 are the only ones under the following conditions. The covering curve has genus $g \leq 9$ and the quotient is a curve of genus $g' \geq 1$, moreover for the case $g'=2$ we extended the calculation up to $g=13$. By Proposition \eqref{xiao} we know that if $g'=1$ these are all the families satisfying $(*)$. It is not bold to conjecture that these are all also in the case $g'>1$. The \verb|MAGMA| script that we used is available at:

  \verb|users.mat.unimi.it/users/penegini/|

  \verb|publications/PossGruppigFix_Elliptic_v2.m|

  \smallskip    

This script differs from the one in \cite{fgp} essentially for the fact that it does not return a representative up to Hurwitz equivalence of a datum. But it gives all possible ramification data. This is because the Hurwitz's moves for the data we have found could be easily handled by hand as we have seen in \ref{say_HWMoves}. In addition, this helped to speed up the finding-example process as well. The other changes in the script are the obvious ones related to the fact that the genus of the base is not $0$ anymore. 
It is important to notice that the \verb|MAGMA| script works perfectly fine for covering curves of genera $g>9$, we simply did not include other results for time reasons. 
\def\cprime{$'$}


\begin{thebibliography}{10}








\bibitem{birman-braids}
J.~S. Birman.
\newblock {\em Braids, links, and mapping class groups}.
\newblock Princeton University Press, Princeton, N.J., 1974.
\newblock Annals of Mathematics Studies, No. 82.




\bibitem{breuer} T.~Breuer.  \newblock {\em Characters and
    automorphism groups of compact {R}iemann surfaces}, volume 280 of
  {\em London Mathematical Society Lecture Note Series}.  \newblock
  Cambridge University Press, Cambridge, 2000.





\bibitem{baffo-linceo} F.~Catanese, M.~L{\"o}nne, and F.~Perroni.
  \newblock Irreducibility of the space of dihedral covers of the
  projective line of a given numerical type.  \newblock {\em Atti
    Accad. Naz. Lincei Cl. Sci. Fis. Mat. Natur. Rend. Lincei (9)
    Mat. Appl.}, 22(3):291--309, 2011.

\bibitem{chenluzuo}
 K.~ Chen, X.~ Lu, K.~ Zuo \newblock A note on Shimura subvarieties in the hyperelliptic Torelli locus.  arXiv:1504.05380. 


\bibitem{CW} C. Chevalley, A. Weil, \textit{\"Uber das Verhalten der
    Intergrale 1. Gattung bei Automorphismen des
    Funktionenkorpers}. Abhand. Math. Sem. Hamburg {\bf 10} (1934),
  358--361.

\bibitem{cf2} E.~Colombo and P.~Frediani.  \newblock Some results on
  the second {G}aussian map for curves.  \newblock {\em Michigan
    Math. J.}, 58(3):745--758, 2009.

\bibitem{cf1} E.~Colombo and P.~Frediani.  \newblock Siegel metric and
  curvature of the moduli space of curves.  \newblock {\em
    Trans. Amer. Math. Soc.}, 362(3):1231--1246, 2010.

\bibitem{cfg} E.~Colombo, P.~Frediani, and A.~Ghigi.  \newblock On
  totally geodesic submanifolds in the {J}acobian locus.  \newblock {\em  Internat. J. Math. }26 (2015), no. 1, 1550005, 21 pp.
  
  
  
  
  
  
  \bibitem{cpt} E.~Colombo, G.~P. Pirola, and A.~Tortora.  \newblock
  Hodge-{G}aussian maps.  \newblock {\em Ann. Scuola Norm. Sup. Pisa
    Cl. Sci. (4)}, 30(1):125--146, 2001.

\bibitem{dejong-noot} J.~de~Jong and R.~Noot.  \newblock Jacobians
  with complex multiplication.  \newblock In {\em Arithmetic algebraic
    geometry ({T}exel, 1989)}, volume~89 of {\em Progr. Math.}, pages
  177--192. Birkh\"auser Boston, Boston, MA, 1991.

\bibitem{dejong-zhang} J.~de~Jong and S.-W. Zhang.  \newblock Generic
  abelian varieties with real multiplication are not {J}acobians.
  \newblock In {\em Diophantine geometry}, volume~4 of {\em CRM
    Series}, pages 165--172. Ed. Norm., Pisa, 2007.


\bibitem{FK} H.~M. Farkas and I.~Kra.  \newblock {\em Riemann
    surfaces}, volume~71 of {\em Graduate Texts in Mathematics}.
  \newblock Springer-Verlag, New York, second edition, 1992.

\bibitem{fgp} P. Frediani, A. Ghigi, M. Penegini  \newblock {\em  Shimura varieties in the Torelli locus via Galois coverings},   International Mathematics Research Notices, volume~20, pages. 10595--10623, (2015). 




\bibitem{gavino} G.~Gonz{\'a}lez~D{\'{\i}}ez and W.~J. Harvey.
  \newblock Moduli of {R}iemann surfaces with symmetry.  \newblock In
  {\em Discrete groups and geometry ({B}irmingham, 1991)}, volume 173
  of {\em London Math. Soc. Lecture Note Ser.}, pages
  75--93. Cambridge Univ. Press, Cambridge, 1992.

\bibitem{grushevsky-moeller-prep} S.~Grushevsky and M.~Moeller.
  \newblock Shimura curves within the locus of genus 3 hyperelliptic
  curves.  \newblock {\tt arXiv:1308.5155 [math.AG]}, 2013.

\bibitem{hain} R.~Hain.  \newblock Locally symmetric families of
  curves and {J}acobians.  \newblock In {\em Moduli of curves and
    abelian varieties}, Aspects Math., E33, pages
  91--108. Friedr. Vieweg, Braunschweig, 1999.

\bibitem{Ha71} W.~J. Harvey.  \newblock On branch loci in
  {T}eichm\"uller space.  \newblock {\em Trans. Amer. Math. Soc.},
  153:387--399, 1971.

\bibitem{kempf-abelian-theta} G.~R. Kempf.  \newblock {\em Complex
    abelian varieties and theta functions}.  \newblock
  Universitext. Springer-Verlag, Berlin, 1991.









\bibitem{kurikuri} I.~Kuribayashi and A.~Kuribayashi.  \newblock
  Automorphism groups of compact {R}iemann surfaces of genera three
  and four.  \newblock {\em J. Pure Appl. Algebra}, 65(3):277--292,
  1990.

\bibitem{liu-yau-ecc} K.~Liu, X.~Sun, X.~Yang, and S.-T. Yau.
  \newblock Curvatures of moduli spaces of curves and applications.
  \newblock {\tt arXiv:1312.6932 [math.DG]}, 2013.  \newblock
  Preprint.

\bibitem{lu-zuo-Mumford-prep} X.~Lu and K.~Zuo.  \newblock The Oort conjecture on Shimura curves in the Torelli locus of curves \newblock {\tt
     	arXiv:1405.4751 [math.AG]}, 2014, \newblock to appear in Compositio Mathematica.

\bibitem{Macl74} 
C. ~Maclachlan. \newblock Modular groups and ber spaces over Teichm\"uller spaces. \newblock {\tt Discontin. Groups Riemann
Surf.}, \newblock  Proc. 1973 Conf. Univ. Maryland,  297--314, 1974.

\bibitem{magaard-e-soci} K.~Magaard, T.~Shaska, S.~Shpectorov, and
  H.~V{\"o}lklein.  \newblock The locus of curves with prescribed
  automorphism group.  \newblock {\em S\=urikaisekikenky\=usho
    K\=oky\=uroku}, (1267):112--141, 2002.  \newblock Communications
  in arithmetic fundamental groups (Kyoto, 1999/2001).

\bibitem{MA} MAGMA Database of Small Groups;
  \verb|http://magma.maths.usyd.edu.au/magma/|
  \verb|htmlhelp/text404.htm|.

\bibitem{M95} R.~Miranda.  \newblock {\em Algebraic curves and
    Mathematics}.  \newblock American Mathematical Society,
 Providence, RI, 1995.

\bibitem{mohazuo} A.~Mohajer, K.~ Zuo.  \newblock On Shimura subvarieties generated by families of abelian covers of $\mathbb{P}^{1}$.  \newblock {\tt arXiv:1402.1900 [math.AG]}, 2014. \newblock Preprint.


\bibitem{moonen-linearity-1} B.~Moonen.  \newblock Linearity
  properties of {S}himura varieties. {I}.  \newblock {\em J. Algebraic
    Geom.}, 7(3):539--567, 1998.

\bibitem{moonen-special} B.~Moonen.  \newblock Special subvarieties
  arising from families of cyclic covers of the projective line.
  \newblock {\em Doc. Math.}, 15:793--819, 2010.


\bibitem{moonen-oort} B.~Moonen and F.~Oort.  \newblock The {T}orelli
  locus and special subvarieties.  \newblock In {\em {H}andbook of
    {M}{oduli: Volume II}}, pages 549--94.  International {P}ress,
  Boston, MA, 2013.

\bibitem{mostow-discontinuous} G.~D. Mostow.  \newblock On
  discontinuous action of monodromy groups on the complex {$n$}-ball.
  \newblock {\em J. Amer. Math. Soc.}, 1(3):555--586, 1988.

\bibitem{oort-can} F.~Oort.  \newblock Canonical liftings and dense
  sets of {CM}-points.  \newblock In {\em Arithmetic geometry
    ({C}ortona, 1994)}, Sympos. Math., XXXVII, pages
  228--234. Cambridge Univ. Press, Cambridge, 1997.

\bibitem{oort-moduli} F.~Oort.  \newblock {M}oduli of abelian
  varieties in mixed and in positive characteristic.  \newblock In
  {\em {H}andbook of {M}{oduli: Volume III}}, pages 75--134.
  International {P}ress, Boston, MA, 2013.


\bibitem{matteo2011} M.~Penegini.  \newblock The classification of
  isotrivially fibred surfaces with {$p_g=q=2$}.  \newblock {\em
    Collect. Math.}, 62(3):239--274, 2011.  \newblock With an appendix
  by S{\"o}nke Rollenske.

\bibitem{penegini2013surfaces} M.~Penegini.  \newblock Surfaces
  isogenous to a product of curves, braid groups and mapping class
  groups.    \newblock  {\em in Beauville Surfaces and Groups},  Springer Proceedings in Math. and Stats. 129--148, 2015.



\bibitem{pietroxiao} G.~P. Pirola.  \newblock On a conjecture of Xiao. J. Reine Angew. Math. 431 (1992), 75--89.





\bibitem{rohde} J.~C. Rohde.  \newblock {\em Cyclic coverings,
    {C}alabi-{Y}au manifolds and complex multiplication}, volume 1975
  of {\em Lecture Notes in Mathematics}.  \newblock Springer-Verlag,
  Berlin, 2009.

\bibitem{S} J.-P. Serre.  \newblock {\em Repr\'esentations lin\'eaires
    des groupes finis}.  \newblock Hermann, Paris, revised edition,
  1978.

\bibitem{shimura-purely-transcendental} G.~Shimura.  \newblock On
  purely transcendental fields automorphic functions of several
  variable.  \newblock {\em Osaka J. Math.}, 1(1):1--14, 1964.

\bibitem{Sw} D. Swinarski;
  http://www.math.uga.edu/~davids/ivrg/CWv2.1.txt

\bibitem{toledo} D.~Toledo.  \newblock Nonexistence of certain closed
  complex geodesics in the moduli space of curves.  \newblock {\em
    Pacific J. Math.}, 129(1):187--192, 1987.

\bibitem{xiao} G.~Xiao.  \newblock Fibered Algebraic Surfeces with Low Slope\newblock {\em
    Math. Annalen}, 276:449--466, 1987.




\end{thebibliography}
\end{document}